\documentclass[12pt,a4paper,reqno]{amsart} %Fr o m detta
\usepackage{amssymb}
\usepackage{latexsym}
\usepackage{amsmath}
\usepackage{mathrsfs}
\usepackage[X2,T1]{fontenc}
\usepackage{cite}
\usepackage{calc}                   %Detta beh?vs f?r att
\newcommand{\scal}[2]{\langle #1,#2\rangle}
\newcommand{\rr}[1]{\mathbf R^{#1}}

\newcommand{\nm}[2]{\Vert #1\Vert _{#2}}
\newcommand{\NM}[2]{\left \Vert #1\right \Vert _{#2}}

\newcommand{\op}{\operatorname{Op}}

\newcommand{\sets}[2]{\{ \, #1\, ;\, #2\, \} }

\newcommand{\ep}{\varepsilon}

\newcommand{\cdo}{\, \cdot \, }

\newcommand{\vrum}{\vspace{0.1cm}}

\newcommand{\GL}{\mathbf{M}}

\newcommand{\nn}[1]{{\mathbf N}^{#1}}

\newcommand{\maclB}{\mathcal B}

\newcommand{\maclK}{\mathcal K}
\newcommand{\maclM}{\mathcal M}
\newcommand{\maclS}{\mathcal S}
\newcommand{\mascB}{\mathscr B}

\newcommand{\mascF}{\mathscr F}

\newcommand{\mascI}{\mathscr I}

\newcommand{\mascN}{\mathscr N}
\newcommand{\mascP}{\mathscr P}

\newcommand{\mascS}{\mathscr S}

\numberwithin{equation}{section}          %Detta g?r att man f?r
\newtheorem{thm}{Theorem}
\numberwithin{thm}{section}

\newcommand{\rubrik}{}
\newtheorem{prop}[thm]{Proposition}

\newtheorem{lemma}[thm]{Lemma}

\theoremstyle{definition}

\newtheorem{defn}[thm]{Definition}

\theoremstyle{remark}
\newtheorem{rem}[thm]{Remark}

\author{Joachim Toft}

\address{Department of Mathematics,
Linn{\ae}us University, V{\"a}xj{\"o}, Sweden}

\email{joachim.toft@lnu.se}

\title{Schatten properties, nuclearity and minimality of shift invariant spaces}

\keywords{Matrices, Modulation spaces, Non-uniform expansions, 
Minimality, Singular values, Nuclear operators}

\subjclass[2010]{primary: 47Bxx, 45P05, 35S05, 42B35
secondary: 65F35, 46A16, 42C15}

\begin{document}

\begin{abstract}
We extend Feichtinger's minimality property on smallest
non-trivial time-frequency shift invariant Banach spaces, to the
quasi-Banach case. Analogous properties are deduced for
certain matrix classes.

\par

We use these results to prove that the pseudo-differential operator
$\op (a)$ is a Schatten-$q$ operator
from $M^\infty$ to $M^p$ and $r$-nuclear operator from $M^\infty$ to $M^r$
when $a\in M^r$ for suitable $p$, $q$ and $r$ in $(0,\infty ]$.
%, as well as some
%characterizations of certain quasi-Banach modulation spaces.
\end{abstract}

\par

\maketitle

%%%%%%%%%%%%%%%%%%%%%%%
\section{Introduction}\label{sec0}
%%%%%%%%%%%%%%%%%%%%%%%

\par

A remarkable property of the weighted Feichtinger algebra $M^1_{(v)}(\rr d)$
is the minimality among non-trivial Banach spaces which are invariant
under time-frequency shifts with respect to the submultiplicative weight $v$.
(See Section \ref{sec1} for notations.)
More precisely, let $\mascB$ be a Banach space which is continuously embedded
in $\mascS '(\rr d)$, the set of tempered distributions on $\rr d$
(or, more generally, in the set $\Sigma _1'(\rr d)$ Gelfand-Shilov distributions of
Beurling type of order $1$ on $\rr d$). If $\mascB$ is invariant
under time-frequency shifts, $f\mapsto e^{i\scal \cdo \xi}f(\cdo -x)$, and satisfies
\begin{align}
\mascB \, {\textstyle{\bigcap}} \, M^1_{(v)}(\rr d) &\neq \{ 0\} \notag
\intertext{and}
\nm {e^{i\scal \cdo \xi}f(\cdo -x)}{\mascB} &\lesssim \nm f{\mascB}\, v(x,\xi ),
\qquad f\in \mascB ,\ x,\xi \in \rr d,
\label{Eq:TFSInv}
\end{align}
%%
%for some constant $C$ which is independent of $f\in \mascB$ and $x,\xi \in \rr d$,
then $M^1_{(v)}(\rr d)\subseteq \mascB$. In fact, the search of non-trivial
smallest translation and modulation invariant Banach space, led Feichtinger
to the discovery of $M^1(\rr d)$ (cf. \cite{Fe1}). Later on, Bonsall deduced
in \cite{Bon1} a slightly different proof compared to \cite{Fe1}, which can also
be found in Section 12.1 in \cite{Gc2}.

\par

In Section \ref{sec2} we extend this minimality property to the case of quasi-Banach spaces.
More precisely,
let $\mascB \subseteq \Sigma _1'(\rr d)$ be a quasi-Banach space such that \eqref{Eq:TFSInv}
holds with a $v$-moderate weight $\omega$ in place of $v$.
% for some constant $C$ which is independent of $f\in \mascB$ and $x,\xi \in \rr d$.
Assume in addition that for some $p\in (0,1]$ we have
\begin{equation}\label{Eq:pTriangle}
\nm {f+g}{\mascB}^p\le \nm f{\mascB}^p+\nm g{\mascB}^p,\qquad f,g\in \mascB ,
\end{equation}
and there is a Gabor frame atom $\psi \in \mascB \, {\textstyle{\bigcap}}\, M^p_{(v)}(\rr d)
\neq \{ 0\}$. Then we deduce $M^p_{(\omega )}(\rr d)\subseteq \mascB$ (see Theorem
\ref{Thm:Minimal}). Note that our restrictions on the involved weights are relaxed compared to
\cite{Bon1,Gc2}.

\par

We mainly follow the approach in \cite{Bon1,Fe1,Gc2}. In particular we
consider spaces of non-uniform Gabor expansions whose coefficients obey
suitable weighted $\ell ^p$-estimates, and prove that these spaces agree with corresponding
modulation spaces. (Cf. Proposition \ref{Prop:MCalMIdent}.)

\par

In Section \ref{sec2} we also deduce analogous minimality properties for certain classes
of matrix operators. In particular, let $J$ be an index set, $\mascB$ be a quasi-Banach space
of matrices $(a(j,k))_{j,k\in J}$ such that the following is true:
\begin{itemize}
\item $A_{j_0,k_0}=(\delta _{j,j_0}\delta _{k,k_0})_{j,k\in J}\in \mascB$ and
$\nm {A_{j_0,k_0}}{\mascB}\le C$ for some constant $C$ which is independent of
$j_0,k_0\in J$;

\vrum

\item $\nm {A_1+A_2}{\mascB}^p\le \nm {A_1}{\mascB}^p+\nm {A_2}{\mascB}^p$.
\end{itemize}
Then we prove that $\mathbb U^p(J)$, the set of all matrices $(a(j,k))_{j,k\in J}$
such that $\sum _{j,k}|a(j,k)|^p<\infty$, is continuously embedded in $\mascB$.

\par

In Sections \ref{sec3} and \ref{sec4} we apply the results in Section \ref{sec2} to deduce
Schatten-von Neumann and nuclear properties for pseudo-differential operators with
symbols in modulation spaces, when acting on (other) modulation spaces.
More precisely, let $a\in M^r_{(\omega )}(\rr {2d})$, $0<r\le 1$. If $r=1$
and all involved weights are trivially equal to $1$ everywhere, then
it is proved in \cite{GH1} that the pseudo-differential operator
$\op (a)$ is continuous from $M^\infty _{(\omega _1)}(\rr d)$ to
$M^r_{(\omega _2)}(\rr d)$, and a
Schatten-von Neumann operator of order $p$ on $L^2(\rr d)=M^2(\rr d)$. The
latter Schatten-von Neumann result was remarked to hold already in \cite{Sj1}.
For general $p\in (0,1]$ and for weighted modulation spaces, these
continuity and Schatten-von Neumann results were extended in \cite{Toft13}.
(See \cite[Theorems 3.1 and 3.4]{Toft13}.)

\par

For certain choices of weights, Theorem 3.4 in \cite{Toft13} was improved in
\cite[Corollary 5.1]{DeRuWa} from which it follows that if $\omega _1=\omega _2$
and $\omega$ are suitable weights, $r\in (0,1]$ and $a\in M^r_{(\omega)}(\rr {2d})$,
then $\op (a)$ is $r$-nuclear on any $M^{p_1,q_1}_{(\omega _1)}(\rr d)$ for any
$p_1,q_1\in [1,\infty ]$. (See also
\cite{DeRu1,DeRu2,DeRu3} for recent progress on $r$-nuclearity of operators
acting on Lebesgue type spaces, \cite{DeRuWa2} for some
extensions to other families of Banach spaces, and \cite{DeRuTo}
for analogous investigations for operators acting on elements defined
on manifolds with boundary.)

\par

In Section \ref{sec3} we apply the minimality properties from Section \ref{sec2}
to deduce that if $p,q,r\in (0,\infty ]$ satisfies
$$
\frac 1r -1\ge \max \left ( \frac 1p-1,0 \right ) + \max
\left ( \frac 1q-1,0 \right ) +\frac 1q,
$$
the weights $\omega _1$, $\omega _2$ and $\omega$
obey the same (relaxed) conditions as in \cite{Toft13} and that
$a\in M^r_{(\omega )}(\rr {2d})$, then $\op (a)$ belongs to
$\mascI _q(M^\infty _{(\omega _1)}(\rr d),M^p_{(\omega _2)}(\rr d))$,
the set of Schatten-von Neumann operators of order $p$ from
$M^\infty _{(\omega _1)}(\rr d)$ to $M^p_{(\omega _2)}(\rr d)$. If
$q=\infty$ and $p\le 1$, then $r=p$, and we recover \cite[Theorems 3.1]{Toft13},
i.{\,}e. $\op (a)$ is continuous from $M^\infty _{(\omega _1)}(\rr d)$ to
$M^r_{(\omega _2)}(\rr d)$ when $a\in M^r_{(\omega )}(\rr {2d})$.
%
%This improve any continuity and compactness property at above, since for any
%$p_j,q_j\in (p,\infty ]$ we have
%$$
%\mascI _p(M^\infty _{(\omega _1)}(\rr d),M^p_{(\omega _2)}(\rr d))
%\subseteq
%\mascI _p(M^{p_1,q_1} _{(\omega _1)}(\rr d),M^{p_2,q_2}_{(\omega _2)}(\rr d)),
%$$
%and the fact that $\op (a)$ is $p$-nuclear on $M^{p_1,q_1}_{(\omega _1)}(\rr d)$
%when it belongs to
%$\mascI _p(M^{p_1,q_1} _{(\omega _1)}(\rr d),
%M^{p_1,q_1}_{(\omega _2)}(\rr d))$. (See \cite{DeRuWa,Olof} and the references
%therein.)

\par

In Section \ref{sec4} the minimality results in Section \ref{sec2}
are applied to deduce $p$-nuclearity for the pseudo-differential operators above.
In fact, let $p\in (0,1]$, $\omega _1$, $\omega _2$ and $\omega$
be as above and let
$a\in M^p_{(\omega )}(\rr {2d})$. Then it is shown that $\op (a)$ belongs to
$\mascN _p(M^\infty _{(\omega _1)}(\rr d),M^p_{(\omega _2)}(\rr d))$,
the set of $p$-nuclear operators from
$M^\infty _{(\omega _1)}(\rr d)$ to $M^p_{(\omega _2)}(\rr d)$. This improves
some of the results in \cite{DeRuWa2} where it is shown that 
$\op (a)$ belongs to
$\mascN _p(M^{p_0,q_0} _{(\omega _1)}(\rr d),M^{p_0,q_0}_{(\omega _2)}(\rr d))$,
provided $p_0,q_0\in [1,\infty]$ and $\omega _1=\omega _2$ are weights of
polynomial type.

\par

By using suitable embedding results between modulation spaces and other types of
function and distribution spaces, the results in Sections \ref{sec3} and \ref{sec4} also
leads to certain Schatten-von Neumann and $r$-nuclearity results in the framework
of such spaces. For example, in these sections we may replace the modulation spaces
by suitable Besov spaces, by using suitable embeddings between such spaces
(see \cite{Ok,SuTo,Toft4} for such embeddings in the Banach space case,
and \cite{WaHu} in the more general quasi-Banach space case).

\par

Finally we remark about recent progresses by J. Delgado, M. Ruzhansky,
N. Togmagambetov and B. Wang on $p$-nuclearity
and Schatten-von Neumann properties in various contexts.
In \cite{DeRu1} optimal conditions on memberships in Schatten-von Neumann classes
over $L^2$ were obtained in terms of Sobolev regularity of the kernels.
In \cite{DeRu2,DeRu3,DeRuWa2}, $p$-nuclearity of operators
acting on Lebesgue type spaces are considered. Especially, in \cite{DeRuWa2} some
conditions for $p$-nuclearity of operators between weighted modulation and
other Banach spaces were obtained in terms of symbols.
Furthermore, in
\cite{DeRuTo} analogous investigations are performed for operators
acting on elements defined on manifolds with boundary.

\par

\section*{Acknowledgement}
I am grateful to H. Feichtinger for encouraging discussions during
his visit at Linn{\ae}us university, autumn 2016. I am also grateful to J. Delgado and
M. Ruzhansky for fruitful discussions on the subject of the paper.

\par

%%%%%%%%%%%%%%%%%%%%%%%
\section{Preliminaries}\label{sec1}
%%%%%%%%%%%%%%%%%%%%%%%

\par

In this section we recall some basic facts of quasi-Banach spaces,
Gelfand-Shilov spaces and modulation spaces.

\par

We start by introducing some notations on quasi-Banach spaces. A quasi-norm
$\nm \cdo{\mascB}$ on a vector space $\mascB$ over
$\mathbf C$ is a non-negative function $\nm \cdo{\mascB}$
on $\mascB$ which is non-degenerate in the sense
$$
\nm f{\mascB}=0\quad \Longleftrightarrow \quad f=0,\qquad f\in \mascB ,
$$
and fulfills
\begin{equation}\label{weaktriangle}
\begin{alignedat}{2}
\nm {\alpha f}{\mascB} &= |\alpha |\cdot \nm f{\mascB}, &\qquad 
f &\in \mascB ,\ \alpha \in \mathbf C
\\[1ex]
\text{and}\qquad \nm {f+g}{\mascB} &\le 2^{\frac 1p-1}(\nm f{\mascB} +\nm g{\mascB}),
&\qquad f,g &\in \mascB , 
\end{alignedat}
\end{equation}
for some constant $p\in (0,1]$ which is independent of $f,g\in \mascB$. Then
$\mascB$ is a topological vector space when the topology for $\mascB$
is defined by $\nm \cdo{\mascB}$, and $\mascB$ is called a quasi-Banach
space if $\mascB$ is complete under this topology.

\par

In general, the conditions \eqref{weaktriangle} does not need to imply that
\begin{equation}\label{Eq:pTriangleIneq}
\nm {f+g}{\mascB}^p \le \nm {f}{\mascB}^p + \nm {g}{\mascB}^p,
\qquad f,g\in \mascB
\end{equation}
should hold. On the other hand, by Aiko-Rolewicz theorem it follows that
we may replace the quasi-norm in
\eqref{weaktriangle} by an equivalent one, and which also satisfies
\eqref{Eq:pTriangleIneq}
(cf. \cite{Aik,Rol}). Here we note that if $\nm \cdo{\mascB}$ is a quasi-norm
which satisfies \eqref{Eq:pTriangleIneq}, then \eqref{weaktriangle}
holds.

\par

From now on we assume that the quasi-norm of $\mascB$
is chosen such that \eqref{weaktriangle} and \eqref{Eq:pTriangleIneq}
hold true.

\par

\subsection{The Gelfand-Shilov space
$\Sigma _1(\rr d)$ and its distribution space}

\par

Next we recall the definition of $\Sigma _1(\rr d)$
and its distribution space $\Sigma _1'(\rr d)$, which belong to the family of
Gelfand-Shilov spaces and their duals (cf. e.{\,}g. \cite{GS}). 
Let $h\in \mathbf R_+$ be fixed. Then $\mathcal S_{1,h}(\rr d)$
is the set of all $f\in C^\infty (\rr d)$ such that
\begin{equation*}%\label{gfseminorm}
\nm f{\mathcal S_{1,h}}\equiv \sup \frac {|x^\beta \partial ^\alpha
f(x)|}{h^{|\alpha | + |\beta |}\alpha !\, \beta !}
\end{equation*}
is finite. Here the supremum is taken over all $\alpha ,\beta \in
\mathbf N^d$ and $x\in \rr d$.

\par

Obviously $\mathcal S_{1,h}\subseteq
\mathscr S$ is a Banach space which increases with $h$.
Furthermore, $\mathcal S_{1,h}$ contains all finite linear combinations of Hermite functions
$h_\alpha (x)=H_\alpha (x)e^{-|x|^2/2}$, where $H_\alpha$ is the Hermite polynomial
of order $\alpha \in \nn d$.
Since such linear combinations are dense in $\mathscr S$, it follows
that the ($L^2$-)dual $(\mathcal S_{1,h})'(\rr d)$ of $\mathcal S_{1,h}(\rr d)$ is
a Banach space which contains $\mathscr S'(\rr d)$.

\par

The \emph{Gelfand-Shilov space} $\Sigma _1(\rr d)$ is the projective limit
of $\mathcal S_{1,h}(\rr d)$ with respect to $h$. This implies that
\begin{equation}\label{GSspacecond1}
\Sigma _{1}(\rr d) =\bigcap _{h>0}\mathcal
S_{1,h}(\rr d)
\end{equation}
is a Fr{\'e}chet space with semi norms $\nm \cdo{\mathcal S_{1,h}}$,
$h>0$.

\medspace

The \emph{Gelfand-Shilov distribution space} $\Sigma _1'(\rr d)$ is the
inductive limit of $\mathcal S_{1,h}'(\rr d)$ with respect to $h>0$.  Hence
\begin{equation}\tag*{(\ref{GSspacecond1})$'$}
\Sigma _1'(\rr d) =\bigcup _{h>0} \mathcal S_{1,h}'(\rr d).
\end{equation}
We remark that $\Sigma _1'(\rr d)$ is the dual of
$\Sigma _1(\rr d)$, also in topological sense (cf. \cite{Pil1}).

\par

From now on we let $\mathscr F$ be the Fourier transform,
given by
$$
(\mathscr Ff)(\xi )= \widehat f(\xi ) \equiv (2\pi )^{-d/2}\int _{\rr
{d}} f(x)e^{-i\scal  x\xi }\, dx
$$
when $f\in L^1(\rr d)$. Here $\scal \cdo \cdo$ denotes the
usual scalar product on $\rr d$. The map $\mathscr F$ extends 
uniquely to homeomorphisms on $\mathscr S'(\rr d)$ and
$\Sigma _1'(\rr d)$, and restricts to homeomorphisms on
$\mathscr S(\rr d)$ and $\Sigma _1(\rr d)$, and to a unitary
operator on $L^2(\rr d)$.

\medspace

Next we recall some mapping properties of Gelfand-Shilov
spaces under short-time Fourier transforms.
Let $\phi \in \mathscr S(\rr d)$ be fixed. For every $f\in
\mathscr S'(\rr d)$, the \emph{short-time Fourier transform} $V_\phi
f$ is the distribution on $\rr {2d}$ defined by the formula
\begin{equation}\label{defstft}
(V_\phi f)(x,\xi ) =\mathscr F(f\, \overline{\phi (\cdo -x)})(\xi ) =
(f,\phi (\cdo -x)e^{i\scal \cdo \xi}).
\end{equation}
We recall that if $T(f,\phi )\equiv V_\phi f$ when $f,\phi \in \Sigma _1(\rr d)$,
then $T$ is uniquely extendable to sequently continuous mappings
\begin{alignat*}{2}
T\, &:\, & \Sigma _1'(\rr d)\times \Sigma _1(\rr d) &\to
\Sigma _1 '(\rr {2d})\bigcap C^\infty (\rr {2d}),
\\[1ex]
T\, &:\, & \Sigma _1'(\rr d)\times \Sigma _1'(\rr d) &\to
\Sigma _1'(\rr {2d})
\end{alignat*}
(cf. \cite{CPRT10,Toft8}).
We also note that $V_\phi f$ takes the form
\begin{equation}\tag*{(\ref{defstft})$'$}
V_\phi f(x,\xi ) =(2\pi )^{-d/2}\int _{\rr d}f(y)\overline {\phi
(y-x)}e^{-i\scal y\xi}\, dy
\end{equation}
for admissible $f$. 

\par

There are several characterizations of Gelfand-Shilov
spaces and their distribution spaces. For example, they can
easily be characterized by Hermite functions and other
related functions (cf. e.{\,}g. \cite{GrPiRo,JaEi}). They can
also be characterized by suitable estimates of their Fourier
and Short-time Fourier transforms (cf.
\cite{ChuChuKim,GrZi,Toft8}).

\par

\subsection{Weight functions}

\par

Next we recall some facts on weight
functions. A \emph{weight} on $\rr d$ is a positive function $\omega
\in  L^\infty _{loc}(\rr d)$ such that $1/\omega \in  L^\infty _{loc}(\rr d)$.
In the sequel we assume that $\omega$ is \emph{moderate},
or \emph{$v$-moderate} for some positive function $v \in
 L^\infty _{loc}(\rr d)$. This means that
\begin{equation}\label{moderate}
\omega (x+y) \lesssim \omega (x)v(y),\qquad x,y\in \rr d.
\end{equation}
Here $A\lesssim B$ means that $A\le cB$
for a suitable constant $c>0$, and for future references, we
write $A\asymp B$
when $A\lesssim B$ and $B\lesssim A$. 
We note that \eqref{moderate} implies that $\omega$ fulfills
the estimates
\begin{equation}\label{moderateconseq}
v(-x)^{-1}\lesssim \omega (x)\lesssim v(x),\quad x\in \rr d.
\end{equation}
We let $\mascP _E(\rr d)$ be the sets of all moderate weights on $\rr d$.

\par

It can be proved that if $\omega \in \mascP _E(\rr d)$, then
$\omega$ is $v$-moderate for some $v(x) = e^{r|x|}$, provided the
positive constant $r$ is chosen large enough (cf. \cite{Gc2.5}). In particular,
\eqref{moderateconseq} shows that for any $\omega \in \mascP
_E(\rr d)$, there is a constant $r>0$ such that
\begin{equation}\label{WeightExpEst}
e^{-r|x|}\lesssim \omega (x)\lesssim e^{r|x|},\quad x\in \rr d.
\end{equation}

\par

We say that $v$ is
\emph{submultiplicative} if $v$ is \emph{even} and \eqref{moderate}
holds with $\omega =v$. In the sequel, $v$ always stand for a
submultiplicative weight if nothing else is stated.

\par

\subsection{Classes of matrices}\label{subsec1.8}

\par

It is suitable for us to consider matrix classes with respect to general
index sets.
%In what follows we let $J$ be an index set, $A$ be the complex matrix
%$(a(j,k))_{j,k\in J}$, $p\in
%(0,\infty ]$ and $\omega$ be a map from $J\times J$
%to $\mathbf R_+$.
%, and
%%%
%\begin{multline}\label{haomegadef}
%h_{A,p,\omega }(k) \equiv \nm {H_{A,\omega}(\cdo ,k)}{\ell ^p},
%\\[1ex]
%\text{where}
%\quad
%H_{A,\omega}(j,k)=a(j,j-k)\omega (j,j-k).
%\end{multline}
%%%

\par

\begin{defn}\label{matrixset1}
Let $p\in (0,\infty ]$, $J_1$ and $J_2$ be index sets, $\boldsymbol J =J_2\times J_1$
and let $\omega$ be a map
from $\boldsymbol J$ to $\mathbf R_+$.
\begin{enumerate}
\item The set $\mathbb U_0'(\boldsymbol J)$ consists of all formal
matrices $(a(j_2,j_1))_{(j_2,j_1)\in \boldsymbol J}$ whose matrix
elements $a(j_2,j_1)$ belongs to $\mathbf C$;

\vrum

\item The set $\mathbb U_0(\boldsymbol J)$ consists of all $A=(a(j_2,j_1))
_{(j_2,j_1)\in \boldsymbol J} \in \mathbb U_0'(\boldsymbol J)$ such that at most
finite numbers of $a(j_2,j_1)$ are non-zero;

\vrum

\item The set $\mathbb U^{p}(\omega ,\boldsymbol J )$ consists of all
$A=(a(j_2,j_1))_{(j_2,j_1)\in \boldsymbol J}\in \mathbb U_0'(\boldsymbol J)$
such that
$$
\nm A{\mathbb U^{p}(\omega ,\boldsymbol J )} \equiv \nm {a\cdot \omega}
{\ell ^p(\boldsymbol J)}
$$
%is finite, where $h_{A,p,\omega }$ is given by \eqref{haomegadef}.
is finite.
%Furthermore,
%$\mathbb U^{p}_0(\omega ,\Lambda )$ is the completion of
%$\mathbb U_0(J )$ under the quasi-norm $\nm {\cdo}{\mathbb
%U^{p}(\omega ,J)}$.
\end{enumerate}
\end{defn}

\par

For conveniency we set  $\mathbb U^{p}(\boldsymbol J)=
\mathbb U^{p}(\omega ,\boldsymbol J)$ when $\omega =1$ everywhere in
Definition \ref{matrixset1}. Furthermore, if $J_1=J_2=J$, then we set $\mathbb U^{p}(\omega ,J)
=\mathbb U^{p}(\omega ,\boldsymbol J)$ and $\mathbb U^{p}(J)
=\mathbb U^{p}(\boldsymbol J)$.

\par

\subsection{Modulation spaces}\label{subsec1.2}

\par

Next we define modulation spaces.
Let $\phi \in \Sigma _1(\rr d)\setminus 0$. For any $p,q\in (0.\infty ]$
and $\omega \in \mascP _E(\rr {2d})$,
the (standard) modulation space $M^{p,q}_{(\omega )}(\rr d)$
is the set of all $f\in \Sigma _1'(\rr d)$ such that $V_\phi f\in
L^{p,q}_{(\omega )}(\rr {2d})$, and we equip $M^{p,q}_{(\omega )}(\rr d)$
with the quasi-norm
\begin{equation}\label{modnorm2}
f\mapsto \nm f{M^{p,q}_{(\omega )}}\equiv \nm {V_\phi f}{L^{p,q}_{(\omega )}}.
\end{equation}
%%
%
%\par
%
%In the same way, the modulation space $M^{p,q}_{(\omega )}(\rr d)$
%of Wiener-amalgam type is the set of all $f\in \maclS _{1/2}'(\rr d)$
%such that $V_\phi f\in L^{p,q}_{*,(\omega )}(\rr {2d})$, and we equip
%$W^{p,q}_{(\omega )}(\rr d)$ with the quasi-norm
%$$
%\nm f{W^{p,q}_{(\omega )}}\equiv \nm {V_\phi f}{L^{p,q}_{*,(\omega )}}.
%$$
For conveniency we also set $M^p_{(\omega )}=M^{p,p}_{(\omega )}$, and
remark that $M^{p,q}_{(\omega )}(\rr d)$ %and $W^{p,q}_{(\omega )}(\rr d)$ are
is one of the most common types of modulation spaces. It was introduced
by Feichtinger in \cite{Fei1} for certain choices of $\omega$. We also set
$M^{p,q}=M^{p,q}_{(\omega)}$ and $M^{p}=M^{p}_{(\omega)}$ when
$\omega =1$.

\par

In the following proposition we list some properties for modulation,
and refer to \cite{Fei1,FG1,GaSa, Gc2,Toft5} for proofs.

\begin{prop}\label{p1.4B}
Let $r\in (0,1]$, $p,p_j,q_j\in (0,\infty ]$ and $\omega ,\omega _j,v\in
\mascP  _{E}(\rr {2d})$, $j=1,2$, be such that $r\le p, p_j,q,q_j$,
$p _1\le p _2$, $q_1\le q_2$,  $\omega _2\lesssim \omega _1$, and
let $\omega$ be $v$-moderate. Then the following is true:
\begin{enumerate}
\item if $\phi \in M^r_{(v)}(\rr d)\setminus 0$, then
$f\in M^{p,q}_{(\omega )}(\rr d)$, if and only if
\eqref{modnorm2} is finite.
In particular, $M^{p,q}_{(\omega )}(\rr d)$ is independent
of the choice of $\phi \in M^r_{(v)}(\rr d)\setminus 0$.
Moreover, $M^{p,q}_{(\omega )}(\rr d)$ is a quasi-Banach
space under the quasi-norm in \eqref{modnorm2}, and different
choices of $\phi$ give rise to equivalent quasi-norms;

\vrum

\item[{\rm{(2)}}] $M^{p _1,q_1}_{(\omega _1)}(\rr d)\subseteq
M^{p _2,q_2}_{(\omega _2)}(\rr d)$.
\end{enumerate}
\end{prop}

\par

\subsection{Gabor analysis and modulation spaces}

\par

Next we define Gabor atoms of certain orders.

\par

\begin{defn}\label{Def:GaborAtom}
Let $v\in \mascP _E(\rr {2d})$ be submultiplicative, $J$ be a countable
set and let $\psi _1\in \Sigma _1'(\rr d)$. Then $\psi _1$ is called a \emph{Gabor atom}
of order $p\in (0,1]$ with respect to $v$, if $\psi _1\in M^p_{(v)}(\rr d)\setminus 0$, and there
exist $\psi _2\in M^p_{(v)}(\rr d)\setminus 0$ and lattices $\{ x_j \} _{j\in J}$
and $\{\xi _k \} _{k\in J}$ in $\rr d$ such that
$$
\big \{ \psi _1(\cdo -x_j)e^{i\scal {\cdo}{\xi _k} } \big \} _{j,k\in J}
\quad \text{and}\quad
\big \{ \psi _2(\cdo -x_j)e^{i\scal {\cdo}{\xi _k} } \big \} _{j,k\in J}
$$
are dual Gabor frames to each others.
\end{defn}

\par

\begin{rem}\label{Rem:GaborAtomOrderOne}
By \cite[Theorem S]{Gc1} it follows that every $\psi \in M^1_{(v)}(\rr d)
\setminus 0$ is a Gabor atom of order $1$.
\end{rem}

\par

\begin{rem}\label{Rem:GaborAtomOrderLessThanOne}
Let $p\in (0,1]$ and $v\in \mascP _E(\rr d)$. By the previous
remark, \cite[Theorem S]{Gc1},
and Remark 1.10 and Theorem 3.7 in \cite{Toft12} it follows that
the set of Gabor atoms of order $p$ contains $\Sigma _1(\rr d)
\setminus 0$. In particular, the set of such atoms is non-empty.

\par

We also remark that by \cite{KSV} it follows that the canonical
dual window of an element in $\Sigma _1(\rr d)\setminus 0$ belongs to
$\Sigma _1(\rr d)$. In fact, let $\mathcal K_1(\rr d)$ be as in \cite{KSV}, then
it is clear that
$\mathcal K_1(\rr d)\cap \mathscr F(K_1(\rr d))= \Sigma _1(\rr d)$
in view of \cite{ChuChuKim}.
By \cite{KSV} it follows that we may choose both $\phi _1$, $\phi _2$
and their dual windows in $\Sigma _1(\rr d)$.
\end{rem}

\par

Assume that $\omega ,v\in \mascP _E(\rr {2d})$,
$p,q\in (0,\infty ]$ and $r\in (0,1]$ are such that $r\le p,q$ and $\omega$
is $v$-moderate, and $\psi _1$, $\psi _2$, $\{x_j \}_{j\in J}$ and $\{ \xi _k\} _{k\in J}$
are the same as in Definition \ref{Def:GaborAtom}. By \cite[Theorem 3.7]{Toft12}
it follows that $f\in \Sigma _1'(\rr d)$ belongs to $M^{p,q}_{(\omega )}(\rr d)$, if and
only if $f$ is given by
$$
f=\sum _{j,k\in J}c_{j,k}e^{i\scal \cdo \xi _k}\psi _1(\cdo -x_j),
$$
where $c_{j,k}= (V_{\psi _2}f)(x_j,\xi _k )$ and satisfies
\begin{equation}\label{Eq:DiscreteQuasiNorm}
\left ( \sum _{k\in J} \left ( \sum _{j\in J}|c_{j,k}
\omega (x_j,\xi _k)|^p\right )^{\frac qp}\right )^{\frac 1p}<\infty .
\end{equation}
Furthermore, the left-hand side of \eqref{Eq:DiscreteQuasiNorm} defines
a quasi-norm in $M^{p,q}_{(\omega )}$, which is equivalent to $f\mapsto
\nm f{M^{p,q}_{(\omega )}}$. (For more facts of such properties, see e.{\,}g.
\cite{FG1,GaSa,Gc2,Rau1,Rau2,Toft12}.)

\par

Next we introduce topological spaces of non-uniform Gabor expansions.

\par

\begin{defn}\label{Def:CalMpDef}
Let $p\in (0,1]$, $v\in \mascP _E(\rr {2d})$ be submultiplicative,
$\omega \in \mascP _E(\rr {2d})$ be $v$-moderate, and let
$\psi \in \Sigma _1'(\rr d)$ be a Gabor atom of order
$p$ with respect to $v$. Then $\maclM ^p_{(\omega )}(\rr d)=
\maclM ^p_{\psi ,(\omega )}(\rr d)$ is the set of all non-uniform Gabor expansions
\begin{equation}\label{Eq:fGaborExp}
f= \sum _{n=0}^\infty a_ne^{i\scal \cdo {\xi _n}}
\psi (\cdo -x_n)
\end{equation}
such that
\begin{equation}\label{Eq:CalMpNormDef}
\nm f{\maclM ^p_{(\omega )}} \equiv \inf \left ( \sum _{n=0}^\infty |a_n\omega (x_n,\xi _n)|^p
\right )^{\frac 1p}
\end{equation}
is finite, where $\sets {(x_n,\xi _n)\in }{n\in \mathbf N}$ is an arbitrary countable set in
$\rr {2d}$. Here the infimum in \eqref{Eq:CalMpNormDef} is taken over all
representatives \eqref{Eq:fGaborExp} of $f$.
\end{defn}

\par

\begin{rem}
By Proposition \ref{Prop:MCalMIdent} in Section \ref{sec2} it follows that $\maclM ^p_{(\omega )}(\rr d)$
is independent of $\psi$ in Definition \ref{Def:GaborAtom}. 
\end{rem}

\par

\subsection{Pseudo-differential operators}

\par

Next we recall some properties in pseudo-differential calculus.
Let $\GL (d,\Omega)$ be the set of $d\times d$-matrices with
entries in the set $\Omega$, $a\in \maclS _s 
(\rr {2d})$, and let $A\in \GL (d,\mathbf R)$ be fixed. Then the
pseudo-differential operator $\op _A(a)$
is the linear and continuous operator on $\Sigma _1 (\rr d)$, given by
\begin{equation}\label{e0.5}
(\op _A(a)f)(x)
=
(2\pi  ) ^{-d}\iint a(x-A(x-y),\xi )f(y)e^{i\scal {x-y}\xi }\,
dyd\xi .
\end{equation}
For general $a\in \Sigma _1'(\rr {2d})$, the
pseudo-differential operator $\op _A(a)$ is defined as the continuous
operator from $\Sigma _1(\rr d)$ to $\Sigma _1'(\rr d)$ with
distribution kernel
\begin{equation}\label{atkernel}
K_{a,A}(x,y)=(2\pi )^{-d/2}(\mascF _2^{-1}a)(x-A(x-y),x-y).
\end{equation}
Here $\mascF _2F$ is the partial Fourier transform of $F(x,y)\in
\Sigma _1'(\rr {2d})$ with respect to the $y$ variable. This
definition makes sense since the mappings
\begin{equation}\label{homeoF2tmap}
\mascF _2\quad \text{and}\quad F(x,y)\mapsto F(x-A(x-y),x-y)
\end{equation}
are homeomorphisms on $\Sigma _1'(\rr {2d})$.
In particular, the map $a\mapsto K_{a,A}$ is a homeomorphism on
$\Sigma _1'(\rr {2d})$.

\par

The standard (Kohn-Nirenberg) representation, $a(x,D)=\op (a)$, and
the Weyl quantization $\op ^w(a)$ of $a$ are obtained by choosing
$A=0$ and $A=\frac 12 I$, respectively, in \eqref{e0.5} and \eqref{atkernel},
where $I$ is the identity matrix

\par

\begin{rem}\label{BijKernelsOps}
By Fourier's inversion formula, \eqref{atkernel} and the kernel theorem
\cite[Theorem 2.2]{LozPerTask} for operators from
Gelfand-Shilov spaces to their duals,
it follows that the map $a\mapsto \op _A(a)$ is bijective from $\Sigma _1'(\rr {2d})$
to the set of all linear and continuous operators from $\Sigma _1(\rr d)$
to $\Sigma _1'(\rr {2d})$.
\end{rem}

\par

By Remark \ref{BijKernelsOps}, it follows that for every $a_1\in \Sigma _1'(\rr {2d})$
and $A_1,A_2\in \GL (d,\mathbf R)$, there is a unique $a_2\in \Sigma _1'(\rr {2d})$ such that
$\op _{A_1}(a_1) = \op _{A_2} (a_2)$. By Section 18.5 in \cite{Ho1},
the relation between $a_1$ and $a_2$
is given by
\begin{equation}
\label{calculitransform}
\op _{A_1}(a_1) = \op _{A_2}(a_2)
\quad \Longleftrightarrow \quad
a_2(x,\xi )=e^{i\scal {(A_1-A_2)D_\xi}{D_x}}a_1(x,\xi ).
\end{equation}
Here we note that the operator $e^{i\scal {AD_\xi}{D_x}}$ is homeomorphic
on $\Sigma _1(\rr {2d})$ and its dual
(cf. \cite{CaTo,CarWal,Tr}). For modulation spaces we have the following
subresult of Proposition 2.8 in \cite{Toft16}.

\par

\begin{prop}\label{Prop:ExpOpSTFT}
Let $s\ge \frac 12$, $A\in \GL (d,\mathbf R)$, $p,q\in (0,\infty ]$,
$\phi ,a\in \Sigma _1(\rr {2d})$ and let $T_A = e^{i\scal{AD_\xi}{D_x}}$.
If $\omega \in \mascP _E(\rr {4d})$ and
$$
\omega _A(x,\xi ,\eta ,y) = \omega (x+Ay,\xi +A^*\eta ,\eta ,y),
$$
then $T_A$ from $\Sigma _1(\rr {2d})$ to $\Sigma _1(\rr {2d})$
extends uniquely to a homeomorphism from $M^{p,q}_{(\omega )}(\rr {2d})$
to $M^{p,q}_{(\omega _A)}(\rr {2d})$, and
\begin{equation}\label{Eq:ExpOpModSp}
\nm {T_Aa}{M^{p,q}_{(\omega _A)}} \asymp
\nm a{M^{p,q}_{(\omega )}}.
\end{equation}
\end{prop}

\par

We also recall that $\op _A(a)$ is a rank-one operator, i.{\,}e.
\begin{equation}\label{trankone}
\op _A(a)f=(2\pi )^{-d/2}(f,f_2)f_1, \qquad f\in \Sigma _1(\rr d),
\end{equation}
for some $f_1,f_2\in \Sigma _1'(\rr d)$,
if and only if $a$ is equal to the \emph{$A$-Wigner distribution}
\begin{equation}\label{wignertdef}
W_{f_1,f_2}^{A}(x,\xi ) \equiv \mascF (f_1(x+A\cdo
)\overline{f_2(x-(I-A)\cdo )} )(\xi ),
\end{equation}
of $f_1$ and $f_2$. If in addition $f_1,f_2\in L^2(\rr d)$, then $W_{f_1,f_2}^{A}$
takes the form
\begin{equation}\label{wignertdef2}
W_{f_1,f_2}^{A}(x,\xi ) = (2\pi )^{-d/2}\int _{\rr d} f_1(x+Ay)
\overline{f_2(x-(I-A)y)}e^{-i\scal y\xi} \, dy.
\end{equation}
(Cf. \cite{BoDoOl1}.) Since the Weyl case is of peculiar interests,
we also set $W_{f_1,f_2}=W_{f_1,f_2}^{A}$
when $A=\frac 12 I$.

\par

\subsection{Schatten-von Neumann classes and nuclear operators}\label{subsec1.4}

\par

Next we recall some Schatten-von Neumann properties of operators, and
start to consider a general situation, involving linear operators from a
(quasi-)Banach space to an other (quasi-)Banach space. (Cf. e.{\,}g. \cite{
Si,BS,Toft11,TB,ToKhNiNo}.)
Let $\mascB _1$ and $\mascB _2$ be (quasi-)Banach spaces and let $T$ be a
linear operator from $\mascB  _1$ to $\mascB _2$. The singular value of
$T$ of order $j\ge 1$ is defined as
$$
\sigma _j(T) = \sigma _j(T;\mascB _1 , \mascB _2 )
\equiv \inf \nm {T-T_0}{\mascB _1\to \mascB _2},
$$
where the infimum is taken over all linear operators $T_0$ from $\mascB _1$ to
$\mascB _2$ of rank at most $j-1$.
The operator $T$ is said to be a Schatten-von Neumann operator of order
$p\in (0,\infty ]$ if
\begin{equation}\label{SchattenNormBanach}
\nm T{\mascI _p(\mascB _1,\mascB _2)} \equiv \nm {\{ \sigma _j(T) \} _{j\ge 1}}{\ell ^p}
\end{equation}
is finite. The set of Schatten-von Neumann operators from $\mascB _1$ to $\mascB _2$
of order $p\in (0,\infty ]$ is denoted by $\mascI _p(\mascB _1,\mascB_2)$.
%, and
%by straight-forward computations it follows that if $\mascB _1=\mascH _1$ and
%$\mascB _2=\mascH _2$ are Hilbert spaces, then the norms in
We observe that $\mascI _p(\mascB _1,\mascB _2)$ is contained in
$\maclK (\mascB _1,\mascB _2)$, the set of compact operators from $\mascB _1$
to $\mascB _2$, when $p<\infty$. Furthermore, $\mascI _\infty (\mascB _1,\mascB _2)$
agrees with $\maclB (\mascB _1,\mascB _2)$, the set of linear bounded operators
from $\mascB _1$ to $\mascB _2$.

\par

If $A\in \GL (d,\mathbf R)$ and
$$
\Sigma _1(\rr d) \subseteq \mascB _1,\mascB _2 \subseteq
\Sigma _1'(\rr d)
$$
with continuous embeddings, then
we let $s_{A,p}(\mascB _1,\mascB_2)$ be the set of all $a\in \Sigma _1'(\rr {2d})$
such that $\op _A(a)\in \mascI _p(\mascB _1,\mascB _2)$, and we set
$$
\nm a{s_{A,p}(\mascB _1,\mascB_2)}\equiv \nm {\op _A(a)}
{\mascI _p(\mascB _1,\mascB_2)}.
$$

\par

Next we define nuclear operators. Let $\mascB _0$ be a Banach space with dual
$\mascB _0'$, $\mascB$ be a quasi-Banach space,
$r\in (0,1]$ and let $T$ be a linear and continuous operator from $\mascB _0$
to $\mascB$. Then $T$ is called \emph{$r$-nuclear} from $\mascB _0$ to
$\mascB$, if there are sequences $\{ \ep _j\} _{j=1}^\infty \subseteq \mascB _0'$
and $\{ e_j\} _{j=1}^\infty \subseteq \mascB$ such that
\begin{align}
&T = \sum _{j=1}^\infty e_j\otimes \ep _j\label{Eq:rNuclDef}
\intertext{with convergence in $\maclB (\mascB _0,\mascB )$, and}
&
\sum _{j=1}^\infty \nm {\ep _j}{\mascB _0'}^r\nm {e_j}{\mascB}^r<\infty .
\label{Eq:NuclQNormEst}
\end{align}
The set of $r$-nuclear operators from $\mascB _0$ to
$\mascB$ is denoted by $\mascN _r(\mascB _0,\mascB )$, and we equip
this set by the quasi-norm
$$
\nm T{\mascN _r(\mascB _0,\mascB )} \equiv \inf 
\left (
\sum _{j=1}^\infty \nm {\ep _j}{\mascB _0'}^r\nm {e_j}{\mascB}^r
\right ) ^{\frac 1r},
$$
where the infimum is taken over all representatives $\{ \ep _j\} _{j=1}^\infty \subseteq \mascB _0'$
and $\{ e_j\} _{j=1}^\infty \subseteq \mascB$ such that \eqref{Eq:rNuclDef} and
\eqref{Eq:NuclQNormEst} hold true.

\par

We note that \eqref{Eq:rNuclDef} is the same as
$$
Tf = \sum _{j=1}^\infty \scal f{\ep _j}e_j.
$$
By straight-forward computations it follows that $\nm \cdo{\mascN _r(\mascB _0,\mascB )}$
is a quasi-norm of order $r$, and that $\mascN _r(\mascB _0,\mascB )$ is complete.
Hence, $\mascN _r(\mascB _0,\mascB )$ is a quasi-norm space of order $r>0$.

\par

Later on we need the following result which shows that $p$-nuclearity is
stable under linear continuous mappings.

\par

\begin{prop}\label{Prop:NuclQBanach}
Let $p,r\in (0,1]$, $\mascB _k$ be quasi-Banach spaces of order
$p$, $\mascB _{0,k}$ be Banach spaces, $k=1,2$, and let
$$
T_1\, :\, \mascB _{0,2}\to \mascB _{0,1}
\quad \text{and}\quad
T_2\, :\, \mascB _1\to \mascB _2.
$$
Then the following is true:
\begin{enumerate}
\item if $T\in \mascI _r(\mascB _{0,1},\mascB _1)$, then
$T_2\circ T\circ T_1\in \mascI _r(\mascB _{0,2},\mascB _2)$,
and
\begin{equation}\label{Eq:SchattComp}
\nm {T_2\circ T\circ T_1}{\mascI _r(\mascB _{0,2},\mascB _2)}
\lesssim \nm {T_1}{\maclB (\mascB _{0,2},\mascB _{0,1})}\nm {T_2}{\maclB (\mascB _1,\mascB _2)}
\nm T{ \mascI _r(\mascB _{0,1},\mascB _1)}\text ;
\end{equation}

\vrum

\item if $T\in \mascN _p(\mascB _{0,1},\mascB _1)$, then
$T_2\circ T\circ T_1\in \mascN _p(\mascB _{0,2},\mascB _2)$,
and
\begin{equation}\label{Eq:NuclComp}
\nm {T_2\circ T\circ T_1}{\mascN _p(\mascB _{0,2},\mascB _2)}
\le \nm {T_1}{\maclB (\mascB _{0,2},\mascB _{0,1})}\nm {T_2}{\maclB (\mascB _1,\mascB _2)}
\nm T{ \mascN _p(\mascB _{0,1},\mascB _1)}.
\end{equation}
\end{enumerate}
\end{prop}

\par

Proposition \ref{Prop:NuclQBanach} is well-known in the literature. For example,
(1) follows immediately from (4.5) and (4.6) in \cite{ChSiTo}. In order to be self-contained
and show some ideas we give a short proof of (2).

\par

\begin{proof}
Let $e_j$ and $\ep _j$ be the same as in Subsection \ref{subsec1.4}
with $\mascB _1$ and $\mascB _{0,1}$ in place of $\mascB _1$ and
$\mascB _{0,1}$, respectively, and let
$f\in \mascB _{0,1}$, $g\in \mascB _{0,2}$.
Then
\begin{align*}
(T\circ T_1)g &= \sum _{j=1}^\infty \scal g{T_1^*\ep _j}e_j 
\quad \text{and}\quad
(T_2\circ T)f = \sum _{j=1}^\infty \scal {f}{\ep _j}T_2e_j 
\intertext{where}
\sum _{j=1}^\infty \nm {T_1^*\ep _j}{\mascB _{0,2}'}^p\nm {e_j}{\mascB _1}^p 
&\le
\nm {T_1}{\maclB (\mascB _{0,2},\mascB _{0,1})}^p
\sum _{j=1}^\infty \nm {\ep _j}{\mascB _{0,1}'}^p\nm {e_j}{\mascB _1}^p.
\intertext{and}
\sum _{j=1}^\infty \nm {\ep _j}{\mascB _{0,1}'}^p\nm {T_2e_j}{\mascB _2}^p 
&\le
\nm {T_2}{\maclB (\mascB _1,\mascB _2)}^p
\sum _{j=1}^\infty \nm {\ep _j}{\mascB _{0,1}'}^p\nm {e_j}{\mascB _1}^p
\end{align*}
The result now follows by combining these estimates and taking the infimum over
all representatives in \eqref{Eq:NuclQNormEst}.
\end{proof}

\par

If $A\in \GL (d,\mathbf R)$ and
$$
\Sigma _1(\rr d) \subseteq \mascB ,\mascB _0 \subseteq
\Sigma _1'(\rr d)
$$
with continuous embeddings, then
we let $u_{A,r}(\mascB _0,\mascB )$ be the set of all $a\in \Sigma _1'(\rr {2d})$
such that $\op _A(a)\in \mascN _r(\mascB _0,\mascB )$, and we set
$$
\nm a{u_{A,r}(\mascB _0,\mascB )}\equiv \nm {\op _A(a)}
{\mascN _r(\mascB _0,\mascB )}.
$$

\par

%%%%%%%%%%%%%%%%%%%%%%%%%%%%%%
\section{Identification and minimization
properties of $M^p_{(v)}$ and $\mathbb U^p(\omega ,J)$,
when $p\in (0,1]$}\label{sec2}
%%%%%%%%%%%%%%%%%%%%%%%%%%%%%%

\par

In this section we show that $\maclM ^p_{(\omega )}(\rr d)$ agrees with
$M^p_{(\omega )}(\rr d)$ when $p\in (0,1]$. We also prove that
$M^p_{(\omega )}(\rr d)$ is minimal among those quasi-Banach
spaces $\mascB \subseteq \Sigma _1'(\rr d)$
which satisfies \eqref{Eq:TFSInv}, \eqref{Eq:pTriangleIneq} and
$\psi \in \mascB$ for some Gabor atom $\psi$ of order $p$ with
respect to $v$. By using similar technique we show analogous minimality properties
of $\mathbb U^p(\omega ,J)$. 

\par

\subsection{Minimality of $M^p_{(\omega )}$}
First we show that $\maclM ^p_{(\omega )}$ is a quasi-Banach space with quasi-norm
\eqref{Eq:CalMpNormDef}.

\par

\begin{prop}\label{CalMpComplete}
Let $p\in (0,1]$, $v\in \mascP _E(\rr {2d})$ be submultiplicative, $\omega \in \mascP
_E(\rr {2d})$ be $v$-moderate and let $\psi \in \Sigma _1'(\rr d)$ be a Gabor atom of order
$p$ with respect to $v$.
Then $\maclM ^p_{(\omega )}$ is a quasi-Banach space with quasi-norm
\eqref{Eq:CalMpNormDef}.
\end{prop}

\par

\begin{proof}
First we prove that $\nm f{\maclM ^p_{(\omega )}}\neq 0$ when
$f\neq 0$.
Choose $(x_n,\xi _n)_{n\ge 0}$ such that \eqref{Eq:fGaborExp} holds. Then
\begin{multline*}
|V_\phi f(x,\xi )\omega (x,\xi )|^p \le \left ( \sum _{n=0}^\infty |a_n|\, 
|V_\phi \psi (x-x_n,\xi -\xi _n)\omega (x,\xi )|\right )^p
\\[1ex]
\le
\sum _{n=0}^\infty |a_n\omega (x_n,\xi _n)|^p
|V_\phi \psi (x-x_n,\xi -\xi _n)v (x-x_n,\xi -\xi _n)|^p
\\[1ex]
\le \nm \psi{M^\infty _{(v)}}\sum _{n=0}^\infty |a_n\omega (x_n,\xi _n)|^p .
\end{multline*}
By taking the supremum over all $(x,\xi )$ and the infimum over all representatives
\eqref{Eq:fGaborExp}, we get
$$
0<\nm {f}{M^\infty _{(\omega )}} \le \nm \psi{M^\infty _{(v)}} \nm f{\maclM ^p_{(\omega )}},
$$
which shows that $\nm f{\maclM ^p_{(\omega )}}=0$, if and only if $f=0$.

\par

Next let $\ep >0$ be arbitrary, $f,g\in \maclM ^p_{(\omega )}(\rr d)$, and choose
representatives \eqref{Eq:fGaborExp} of $f$ and
$$
g= \sum _{n=0}^\infty b_ne^{i\scal \cdo {\eta _n}}
\psi (\cdo -y_n)
$$
such that
\begin{align*}
\sum _{n=0}^\infty |a_n\omega (x_n,\xi _n)|^p
&\le \nm f{\maclM ^p_{(\omega )}}^p +\frac \ep 2
\intertext{and}
\sum _{n=0}^\infty |b_n\omega (y_n,\eta _n)|^p
&\le \nm g{\maclM ^p_{(\omega )}}^p+\frac \ep 2 .
\end{align*}
This gives,
\begin{multline*}
\nm {f+g}{\maclM ^p_{(\omega )}} ^p
\le
\sum _{n=0}^\infty \Big (|(a_n\omega (x_n,\xi _n)|^p+|b_n\omega (y_n,\eta _n)|^p
\Big )
\\[1ex]
%\le
%2^{\frac 1p-1}\left (
%\left ( \sum _{n=0}^\infty |a_n\omega (x_n,\xi _n)|^p
%\right )^{\frac 1p}
%+
%\left ( \sum _{n=0}^\infty |b_n\omega (x_n,\xi _n)|^p
%\right )^{\frac 1p}
%\right )
%\\[1ex]
\le \nm f{\maclM ^p_{(\omega )}}^p +\nm g{\maclM ^p_{(\omega )}}^p + \ep .
\end{multline*}
Hence,
$$
\nm {f+g}{\maclM ^p_{(\omega )}}^p \le \nm f{\maclM ^p_{(\omega )}}^p
+\nm g{\maclM ^p_{(\omega )}}^p,
$$
since $\ep >0$ was arbitrarily chosen. This implies that $\maclM ^p_{(\omega )}$
is a quasi-normed space of order $p$.

\par

The completeness follows by standard arguments. More precisely, %let $\ep >0$
%be arbitrary and
let $\{ f_m \} _{m=1}^\infty$ be a Cauchy sequence in
$\maclM ^p_{(\omega )}(\rr d)$. Then there is an increasing
sequence, $\{ m_k\} _{k\ge 1}$, of positive integers such that
$$
\nm {f_{m_{k}}-f_{m_{k-1}}}{\maclM ^p_{(\omega )}}^p\le 2^{-k},\qquad k\ge 2.
$$
For every $k\ge 0$, there are sequences
$$
\{ a_{n,k} \} _{n=0}^\infty \subseteq \mathbf C,
\quad
\{ x_{n,k} \} _{n=0}^\infty \subseteq \rr d
\quad \text{and}\quad 
\{ \xi _{n,k} \} _{n=0}^\infty \subseteq \rr d
$$
such that
\begin{align*}
\sum _{n=0}^\infty a_{n,k}e^{i\scal \cdo {\xi _{n,k}}}\psi (\cdo -x_{n,k})
&=
{
\begin{cases}
f_{m_1},\quad &k=1
\\[1ex]
f_{m_k}-f_{m_{k-1}},\quad &k\ge 2
\end{cases}
}
\intertext{and}
\sum _{n=0}^\infty |a_{n,k}\omega (x_{n,k},\xi _{n,k})|^p
&=
{
\begin{cases}
\nm {f_{m_1}}{\maclM ^p_{(\omega )}}^p+1,\quad &k=0
\\[1ex]
\nm {f_{m_k}-f_{m_{k-1}}}{\maclM ^p_{(\omega )}}^p+2^{-k},\quad &k\ge 1 .
\end{cases}
}
\end{align*}

\par

Now let $\{ (c_n,z_n,\zeta _n)\} _{n=0}^\infty$ be an enumeration of
$\{ (a_{n,k},x_{n,k},\xi _{n,k})\} _{n,k=0}^\infty$, and let
$$
f=\sum _{n=0}^\infty c_ne^{i\scal \cdo {\zeta _n}}\psi (\cdo -z_n).
$$
Then $\nm f{\maclM ^p_{(\omega )}}<\infty$, and
\begin{multline*}
\nm {f-f_{m_k}}{\maclM ^p_{(\omega )}}^p 
=
\NM{\sum _{j=k+1}^\infty (f_{m_j}-f_{m_{j-1}})}{\maclM ^p_{(\omega )}}^p
\\[1ex]
\le
\sum _{j=k+1}^\infty \sum _{n=0}^\infty |a_{n,j}(x_{n,j},\xi _{n,j})|^p
\le
\sum _{j=k+1}^\infty
\left ( \nm {f_{m_j}-f_{m_{j-1}})}{\maclM ^p_{(\omega )}}^p +2^{-j}\right )
\\[1ex]
\le
2\sum _{j=k+1}^\infty 2^{-j}\to 0
\end{multline*}
%%
%
%
%
%
%
%\le 2\sum _{j=k}^\infty 2^{-j}\to 0
%$$
as $k$ tends to $\infty$. This in turn gives
$$
\nm {f-f_m}{\maclM ^p_{(\omega )}}^p
\le
\nm {f-f_{m_k}}{\maclM ^p_{(\omega )}}^p +
\nm {f_{m_k}-f_m}{\maclM ^p_{(\omega )}}^p \to 0 ,
$$
as $m$ and $k$ tends to $\infty$, and the completeness of
$\maclM ^p_{(\omega )}(\rr d)$ follows.
\end{proof}

\par

\begin{prop}\label{Prop:MCalMIdent}
Let $p\in (0,1]$, $v\in \mascP _E(\rr {2d})$ be submultiplicative
$\omega \in \mascP _E(\rr {2d})$ be $v$-moderate
and let $\psi \in \Sigma _1'(\rr d)$ be a Gabor atom of order
$p$ with respect to $v$. Then $\maclM ^p_{\psi ,(\omega )}(\rr d)
= M^p_{(\omega )}(\rr d)$ with
equivalent quasi-norms.
\end{prop}

\par

For the proof we need the following lemma.

\par

\begin{lemma}\label{Lemma:TransMod}
Let $p\in (0,1]$, $v\in \mascP _E(\rr {2d})$ be submultiplicative,
$\omega \in \mascP _E(\rr {2d})$ be
$v$-moderate, $\psi \in M^p_{(\omega )}(\rr d)$, and set
$$
\psi _X= e^{i\scal \cdo \xi}\psi (\cdo -x)
\quad \text{and}\quad
\omega _X(y,\eta ) = \omega (y-x,\eta -\xi )
$$
when $X=(x,\xi )\in \rr {2d}$. Then $\omega _X$ is $v$-moderate
for every $X\in \rr {2d}$, and $\nm {\psi _X}{M^p_{(\omega _X)}}$ is independent
of $X\in \rr {2d}$, when the window function of the modulation space
norm is fixed.
\end{lemma}

\par

\begin{proof}
The fact that  $\omega _X$ is $v$-moderate is a straight-forward consequence of the
fact that $v$ is submultiplicative. The details are left for the reader.

\par

We have
$$
|V_{\phi}\psi _X \omega _X|=|V_{\phi}\psi (\cdo -X) \omega (\cdo -X)|,
$$
and the $X$-independency of $\nm {\psi _X}{M^p_{(\omega _X)}}$
follows by applying the $L^p$ quasi-norm on the last equality.
\end{proof}

\par

\begin{proof}[Proof of Proposition \ref{Prop:MCalMIdent}]
By \cite[Theorem 3.7]{Toft12}, it follows that $M^p_{(\omega )}$ is continuously embedded in
$\maclM ^p_{(\omega )}$.

\par

We need to prove the opposite embedding. Let $\maclM _0(\rr d)$
be the set of all expansions in \eqref{Eq:fGaborExp} such that at most finite
numbers of $a_n$ are non-zero. By straight-forward arguments of approximations
it follows that $\maclM_0(\rr d)$ is contained and dense in both $M^p_{(\omega )}(\rr d)$
and $\maclM ^p_{(\omega )}(\rr d)$. The result therefore follows if we prove
\begin{equation}\label{Eq:RevIneqMpCalMp}
\nm f{M^p_{(\omega )}} \lesssim \nm f{\maclM ^p_{(\omega )}}
\end{equation}
when $f\in \maclM _0(\rr d)$.

\par

Assume that $f\in \maclM _0(\rr d)$, let $\ep >0$ and let $\phi$ be Gaussian,
and choose a representation \eqref{Eq:fGaborExp} such that
$$
\left ( \sum _{n=0}^\infty |a_n\omega (x_n,\xi _n)|^p
\right )^{\frac 1p} \le \nm f{\maclM ^p_{(\omega )}} +\ep .
$$
By straight-forward arguments of approximations, we may assume that $a_n$ are
non-zero only for finite numbers of $n$.
Also let $\psi _n =\psi _{X_n}$ and $\omega _n = \omega _{X_n}$, $X_n=(x_n,\xi _n)$,
where $\psi _X$ and $\omega _X$ are the same as in Lemma \ref{Lemma:TransMod}.
Then there is a lattice $\Lambda =\Lambda _1\times \Lambda _2$, where
$\Lambda _1,\Lambda _2\subseteq \rr d$, and such that
\begin{multline*}
\nm f{M^p_{(\omega )}}
\asymp
\left ( \sum _{(j,\iota )\in \Lambda}\left | \sum _{n=1}^\infty
a_nV_{\phi}(e^{i\scal \cdo {\xi _n}} \psi (\cdo -x_n))(j,\iota )
\omega (j,\iota )\right |^p \right )^{\frac 1p}
\\[1ex]
\le
\left ( \sum _{(j,\iota )\in \Lambda}\sum _{n=1}^\infty \left |
a_nV_{\phi}(e^{i\scal \cdo {\xi _n}} \psi (\cdo -x_n))(j,\iota )
\omega (j,\iota )\right |^p \right )^{\frac 1p}
\\[1ex]
\le
\left ( \sum _{n=1}^\infty |a_n\omega (x_n,\xi _n)|^p
\left ( \sum _{(j,\iota )\in \Lambda}|V_\phi \psi _n(j,\iota )v_n(j,\iota ) |^p
\right ) \right )^{\frac 1p}
\\[1ex]
\asymp
\left ( \sum _{n=1}^\infty |a_n\omega (x_n,\xi _n)|^p
\nm{\psi _n}{M^p_{(v_n)}}^p
\right )^{\frac 1p}
\\[1ex]
=
\nm{\psi }{M^p_{(v)}}
\left ( \sum _{n=1}^\infty |a_n\omega (x_n,\xi _n)|^p
\right )^{\frac 1p}
\le
\nm{\psi }{M^p_{(v)}} \big ( \nm f{\maclM ^p_{(\omega )}} +\ep  \big ).
\end{multline*}

\par

Since $\ep$ is chosen arbitrarily, \eqref{Eq:RevIneqMpCalMp} follows.
\end{proof}

\par

By the previous proposition it follows that $\maclM ^p_{\psi ,(\omega)}(\rr d)$
is independent of the choice of $\psi$, which justifies the usage of the notation
$\maclM ^p_{(\omega )}$ instead of $\maclM ^p_{\psi ,(\omega )}$ above.

\par

We are now prepared to formulate and prove the extension of Feichtinger's minimization
property in \cite{Fe1} to the case of quasi-Banach spaces.

\par

\begin{thm}\label{Thm:Minimal}
Let $\omega ,v\in \mascP _E(\rr {2d})$ be such that $\omega$ is $v$-moderate,
$p\in (0,1]$, $\mascB \subseteq \Sigma _1'(\rr d)$ be a quasi-Banach space such
that the following is true:
\begin{enumerate}
\item the quasi-norm $\nm \cdo{\mascB}$ of $\mascB$
satisfies \eqref{Eq:pTriangleIneq};

\vrum

\item $\mascB$ is invariant under time-frequency shifts
$f\mapsto e^{i\scal \cdo \xi}f(\cdo -x)$, and
$$
\nm {e^{i\scal \cdo \xi}f(\cdo -x)}{\mascB} \lesssim \omega (x,\xi )\nm f{\mascB} \text ;
$$

\vrum

\item $\mascB$ contains a Gabor atom of order $p$ with respect to $v$.
\end{enumerate}

\par

Then $M^p_{(\omega )}(\rr d)$ is continuously embedded in $\mascB$.
\end{thm}

\par

If $p=1$ and $\omega =v$, then Theorem \ref{Thm:Minimal} (3) is equivalent to
$\mascB \bigcap M^1_{(v)}(\rr d) \neq \{ 0\}$ in view of \cite[Theorem S]{Gc1}. Hence,
Theorem \ref{Thm:Minimal} extends the Feichtinger's minimizing property.

\par

After the previous preparations, the proof is essentially the same as in the
Banach space case, $p=1$. In order to be self-contained we here present the arguments.

\par

\begin{proof}
We may assume that $\maclM _0(\rr d)$ is dense in $\mascB$, where 
$\maclM _0$ is the same as in the proof of Proposition \ref{Prop:MCalMIdent}.
By the assumptions, there is a Gabor atom $\psi$ of order $p$ with respect to
$v$. By
Proposition \ref{Prop:MCalMIdent}, $M^p_{(\omega )}(\rr d)$ consists of all
$f$ in \eqref{Eq:fGaborExp} which satisfies \eqref{Eq:CalMpNormDef}.

\par

Let $f\in \maclM _0(\rr d)$. Then
\begin{multline*}
\nm f{\mascB} ^p\le \left \Vert \sum _{n=1}^\infty a_n e^{i\scal \cdo {\xi _n}}
\psi (\cdo -x_n) \right \Vert _{\mascB}^p
\\[1ex]
\le
\sum _{n=1}^\infty |a_n|^p \nm {e^{i\scal \cdo {\xi _n}} \psi (\cdo -x_n)}{\mascB}^p
\lesssim
\sum _{n=1}^\infty |a_n\omega (x_n,\xi _n)|^p.
\end{multline*}
By taking the infimum of the right-hand side, we obtain
$$
\nm f{\mascB}\lesssim \nm f{M^p_{(\omega )}},
$$
and the result follows.
\end{proof}

\par

We also have corresponding maximality property of Theorem \ref{Thm:Minimal}
of translation and modulation invariant spaces.

\par

\begin{thm}\label{Thm:Maximality}
Let $v\in \mascP _E(\rr {2d})$ be submultiplicative, $\mascB
\subseteq \Sigma _1'(\rr d)$ be a Banach space such that the
following is true:
\begin{enumerate}
\item $\mascB$ is invariant under time-frequency shifts,
$f\mapsto e^{i\scal \cdo \xi}f(\cdo -x)$, and
$$
\nm {e^{i\scal \cdo \xi}f(\cdo -x)}{\mascB} \lesssim v(x,\xi )\nm f{\mascB} \text ;
$$

\vrum

\item $\mascB$ contains a Gabor atom of order $1$.
\end{enumerate}

\par

Then $\mascB$ is continuously embedded in $M^\infty _{(1/v)}(\rr d)$.
\end{thm}

\par

We need some preparations for the proof. We note that
$\Sigma _1(\rr d)\subseteq \mascB$ in view of Theorem
\ref{Thm:Minimal}. Let $\nm \phi{\mascB '}$ be the dual norm
of $\phi \in \Sigma _1(\rr d)$
with respect to the $L^2$ form be defined by
$$
\nm \phi {\mascB '}\equiv \sup |(f,\phi )_{L^2}|,
$$
where the supremum is taken over all $f\in \mascB$ such that
$\nm f{\mascB}\le 1$, and let the $L^2$-dual $\mascB '$
of $\mascB$ be the completion of $\Sigma _1(\rr d)$ under this
norm (cf. \cite{Toft11}).

\par

\begin{proof}[Proof of Theorem \ref{Thm:Maximality}]
Let $\mascB '$ be the $L^2$ dual of $\mascB$, and let $\Omega$
be the set of all $f\in \mascB$ such that $\nm f{\mascB}\le 1$.
Then $\mascB '$ contains at least one element in $\Sigma _1$, and
if $\phi \in \Sigma _1(\rr d)$, we get
\begin{multline*}
\nm {\phi (\cdo -x)e^{i\scal \cdo \xi}}{\mascB '}
\equiv
\sup _{f\in \Omega} \left (| (f,\phi (\cdo -x)e^{i\scal \cdo \xi })| \right )
\\[1ex]
=
\sup _{f\in \Omega} \left (| (f(\cdo +x)e^{-i\scal \cdo \xi } ,\phi )| \right )
\le
\sup _{f\in \Omega} \left (\nm {f(\cdo +x)e^{-i\scal \cdo \xi }}{\mascB}\nm \phi {\mascB '}
\right )
\\[1ex]
\lesssim
\nm \phi {\mascB '}v(-x,-\xi )
=
\nm \phi {\mascB '}v(x,\xi ).
\end{multline*}
Hence $\mascB '$ is translation and modulation invariant.

\par

By Theorem \ref{Thm:Minimal} it follows that
$$
\Sigma _1(\rr d)\subseteq M^1_{(v)}(\rr d) \subseteq \mascB '.
$$
This gives
$$
\nm f{M^\infty _{(1/v)}} \asymp \sup _{\nm \phi {M^1_{(v)}} \le 1} |(f,\phi )_{L^2}|
\lesssim
\sup _{\nm \phi {\mascB '} \le 1} |(f,\phi )_{L^2}|
\le
\nm f{\mascB},
$$
and the result follows.
\end{proof}

\par

We also have the following characterization of certain modulation spaces. Here let
$M^p_{0,(v)}(\rr {d_1+d_2})$ be the set of all
\begin{equation}\label{Eq:FTensorSum}
F=\sum _{j=1}^\infty f_{1,j}\otimes f_{2,j}
\end{equation}
such that
\begin{equation}\label{Eq:FTensorSumEst}
\sum _{j=1}^\infty
\nm {f_{1,j}}{M^p_{(v_1)}}^p
\nm {f_{2,j}}{M^p_{(v_2)}}^p <\infty ,
\end{equation}
where 
\begin{equation}\label{Eq:TensorvDef}
v(x_1,x_2,\xi _1,\xi _2)=v_1(x_1,\xi _1)v_2(x_2,\xi _2),
\end{equation}
and $v_k\in \mascP _E(\rr{2d_k})$ are submultiplicative, $k=1,2$.
We equip $M^p_{0,(v)}(\rr {d_1+d_2})$ with the norm
$$
F\mapsto \nm F{M^p_{0,(v)}}
\equiv
\inf \left (
\sum _{j=1}^\infty
\nm {f_{1,j}}{M^p_{(v_1)}}^p
\nm {f_{2,j}}{M^p_{(v_2)}}^p
\right )^{\frac 1p},
$$
where the infimum is taken over all representatives $\{ f_{k,j} \}_{j=1}^\infty
\subseteq M^p_{(v_k)}(\rr {d_k})$, $k=1,2$ and such that \eqref{Eq:FTensorSum}
and \eqref{Eq:FTensorSumEst} hold.

\par

\begin{prop}\label{Prop:ModEquivConseq}
Let $v_k\in \mascP _E(\rr{2d_k})$ be submultiplicative, $k=1,2$, let $p\in (0,1]$,
and let $v$ be given by \eqref{Eq:TensorvDef}. Then $M^p_{0,(v)}(\rr {d_1+d_2})
=M^p_{(v)}(\rr {d_1+d_2})$ with equivalent quasi-norms.
\end{prop}

\par

We remark that proofs of Proposition \ref{Prop:ModEquivConseq} in the special case
$p=1$ have been demonstrated by H. Feichtinger in different occasions. The proof
of the general case here below is based on arguments used by H. Feichtinger in his proofs of
the case $p=1$.

\par

\begin{proof}
Since $M^p_{(v)}(\rr {d_1+d_2})$ is a quasi-Banach space of order $p$, we get
\begin{equation*}
%\begin{multline*}
\NM {\sum _{j=1}^\infty f_{1,j}\otimes f_{2,j}}{M^p_{(v)}}^p
\le
\sum _{j=1}^\infty \nm {f_{1,j}\otimes f_{2,j}}{M^p_{(v)}}^p
%\\[1ex]
=
\sum _{j=1}^\infty
\nm {f_{1,j}}{M^p_{(v_1)}}^p
\nm {f_{2,j}}{M^p_{(v_2)}}^p,
\end{equation*}
%\end{multline*}
%%
giving that $M^p_{0,(v)}(\rr {d_1+d_2})
\subseteq M^p_{(v)}(\rr {d_1+d_2})$.

\par

On the other hand, it is clear that $M^p_{0,(v)}(\rr {d_1+d_2})$
contains a Gabor atom of order $p$ with respect to $v$, and that
$$
\nm {e^{i\scal \cdo \xi}F(\cdo -x)}{M^p_{0,(v)}}\lesssim v(x,\xi )\nm F{M^p_{0,(v)}},
\qquad x,\xi \in \rr {d_1+d_2}
$$
when $F\in M^p_{0,(v)}(\rr {d_1+d_2})$. Furthermore, in view of the proof of
Proposition \ref{CalMpComplete} it follows that
$M^p_{0,(v)}(\rr {d_1+d_2})$ is complete, and thereby is a quasi-Banach space
of order $p$. 

\par

By Theorem \ref{Thm:Minimal} it now follows that $M^p_{(v)}(\rr {d_1+d_2})$
is continuously embedded in $M^p_{0,(v)}(\rr {d_1+d_2})$. Hence
$M^p_{0,(v)}(\rr {d_1+d_2})= M^p_{(v)}(\rr {d_1+d_2})$, and the result follows.
\end{proof}

\par

\subsection{Minimality of $\mathbb U^p(\omega ,\boldsymbol J)$}

\par

The following result is the matrix version of
Theorem \ref{Thm:Minimal}.

\par

\begin{prop}\label{Prop:MatrixMinimal}
Let $p\in (0,1]$, $J_1$ and $J_2$ be index sets, $\boldsymbol J=J_2\times J_1$,
$\omega$ be a positive function on $\boldsymbol J$, $\mascB \subseteq \mathbb
U_0'(\boldsymbol J)$
be a quasi-Banach space such that the following conditions hold true:
\begin{enumerate}
\item $\nm {A_1+A_2}{\mascB}^p \le
\nm {A_1}{\mascB}^p+\nm {A_2}{\mascB}^p$
when $A_1,A_2\in \mascB$;

\vrum

\item $A_{k_{2},k_{1}}\equiv (\delta _{j_1,k_{1}}\delta _{j_2,k_{2}})_{(j_2,j_1)\in \boldsymbol J}$
belongs to $\mascB$ for every $(k_{2},k_{1})\in \boldsymbol J$, and
$$
\nm {A_{k_2,k_1}}{\mascB}\le C\omega (k_2,k_1),
$$
for some constant $C>0$ which is independent of $(k_2,k_1)\in \boldsymbol J$.
\end{enumerate}
Then $\mathbb U^p(\omega ,\boldsymbol J)$ is continuously embedded in $\mascB$.
\end{prop}

\par

\begin{proof}
Since $\mathbb U_0(\boldsymbol J)$ is dense in $\mathbb U^p(\omega ,\boldsymbol J)$ and
$\mathbb U_0(\boldsymbol J)\subseteq \mascB$ by the assumptions, it suffices to prove
$$
\nm A{\mascB}\le C\nm A{\mathbb U^p(\omega ,\boldsymbol J)},\quad A\in
\mathbb U_0(\boldsymbol J).
$$

\par

If $A\in \mathbb U_0(\boldsymbol J)$, then
\begin{multline*}
\nm A{\mascB}^p = \NM {\sum _{j_1,j_2}a(j_2,j_1)A_{j_2,j_1}}{\mascB}^p
\le
\sum _{j_1,j_2}|a(j_2,j_1)|^p\nm {A_{j_2,j_1}}{\mascB}^p
\\[1ex]
\le
C^p\sum _{j_1,j_2}|a(j_2,j_1)|^p\omega (j_2,j_1)^p = C^p\nm A{\mathbb
U^p(\omega ,\boldsymbol J)}^p,
\end{multline*}
and the result follows.
\end{proof}

\par

%%%%%%%%%%%%%%%%%%%%%%%
\section{Schatten-von Neumann properties for
operators with kernels in modulation spaces}
\label{sec3}
%%%%%%%%%%%%%%%%%%%%%%%

\par

In this section we use results from the previous section to
deduce Schatten-von Neumann properties of operators with kernels in
$M^p_{(\omega )}$. At the same time we deduce analogous properties for
pseudo-differential operators with symbols in $M^p_{(\omega )}$.

\par

More precisely, we have the following.

\par

\begin{thm}\label{Thm:QuasiBanachSchatten}
Let $A\in \GL (d,\mathbf R)$, $p,q,r\in (0,\infty ]$ be such that
\begin{equation}\label{Eq:pqrCond}
\frac 1r -1\ge \max \left ( \frac 1p-1,0 \right ) + \max
\left ( \frac 1q-1,0 \right ) +\frac 1q ,
\end{equation}
and let
$\omega _0\in \mascP (\rr {4d})$ and $\omega _1,\omega _2
\in \mascP _E(\rr {2d})$ be such that
\begin{equation}\label{Eq:omega120Cond}
\frac {\omega _2(x,\xi  )}{\omega _1
(y,\eta )} \lesssim \omega _0( x+A(y-x) ,\eta +A^*(\xi -\eta ),\xi -\eta ,y-x )
\end{equation}
holds true. Then
\begin{equation}\label{Eq:SchattenPseudo}
M^r_{(\omega _0)}(\rr {2d})\subseteq 
s_{A,q} (M^\infty _{(\omega _1)}(\rr d),M^p_{(\omega _2)}(\rr d) ).
\end{equation}
\end{thm}

\par

The proof of Theorem \ref{Thm:QuasiBanachSchatten} is based on
analogous Schatten-von Neumann properties of matrix operators.

\par

\begin{prop}\label{Prop:MatrixOpSchatten}
Let $p,q,r\in (0,\infty ]$ be such that \eqref{Eq:pqrCond} holds,
$J_k$ be index sets, $\omega _k$
be positive functions on $J_k$, $k=1,2$, let $\boldsymbol J = J_2\times J_1$ and suppose
$$
\frac {\omega _2(j_2)}{\omega _1(j_1)}\lesssim \omega (j_2,j_1),\quad (j_2,j_1)\in \boldsymbol J.
$$
Then $\mathbb U^r(\omega ,\boldsymbol J)\subseteq \mascI _q
(\ell ^\infty _{(\omega _1)}(J_1),\ell ^p_{(\omega _2)}(J_2))$.
\end{prop}

\par

We need the following lemma for the proof of Proposition
\ref{Prop:MatrixOpSchatten}.

\begin{lemma}\label{Lemma:pSchattenTriangle}
Let $p,q\in (0,\infty ]$, and let $\mascB _1$ and $\mascB _2$ be quasi-Banach spaces
such that \eqref{Eq:pTriangle} holds with $\mascB _2$ in place of $\mascB$. Then
%%%
%\begin{align}
%\sigma _{j_1+j_2+1}(T_1+T_2;\mascB _1,\mascB _2)
%&\le 2^{\max (\frac 1p-1,0)}(\sigma _{j_1+1}(T_1;\mascB _1,\mascB _2)
%+\sigma _{j_2+1}(T_2;\mascB _1,\mascB _2))
%\intertext{and}
%\nm {T_1+T_2}{\mascI _q(\mascB _1,\mascB_2)}
%&\le
%2^{\max (\frac 1p-1,0)+\frac 1q}
%\nm {T_1}{\mascI _q(\mascB _1,\mascB_2)}
%+
%\nm {T_2}{\mascI _q(\mascB _1,\mascB_2)}
%\end{align}
%%%
%%
\begin{align}
\sigma _{j_1+j_2+1}(T_1+T_2)
&\le 2^{\max (\frac 1p-1,0)}(\sigma _{j_1+1}(T_1)
+\sigma _{j_2+1}(T_2))\label{Eq:SingQuBTriang}
\intertext{and}
\nm {T_1+T_2}{\mascI _q(\mascB _1,\mascB_2)}
&\le
2^{\max (\frac 1p-1,0)+\max (\frac 1q-1,0)+\frac 1q}
(\nm {T_1}{\mascI _q(\mascB _1,\mascB_2)}
+
\nm {T_2}{\mascI _q(\mascB _1,\mascB _2)})\label{Eq:ShattQuBTriang}
\end{align}
when $T_1,T_2\in \maclB (\mascB _1,\mascB _2)$.
\end{lemma}

\par

\begin{proof}
Let $T_1,T_2, T_{1,j},T_{2,j}\in \maclB (\mascB _1,\mascB _2)$
be such that $T_{1,j}$ and $T_{2,j}$ are operators of rank at most $j-1$,
and let $f\in \mascB_1$. Then
\begin{multline*}
\nm {(T_1+T_2-(T_{1,j_1}+T_{2,j_2}))f}{\mascB _2}
\\[1ex]
\le
2^{\max (\frac 1p-1,0)}\nm {(T_1-T_{1,j_1})f}{\mascB _2}
+ \nm {(T_2-T_{2,j_2})f}{\mascB _2},
\end{multline*}
which gives
\begin{multline*}
\nm {T_1+T_2-(T_{1,j}+T_{2,j})}{\maclB (\mascB _1,\mascB _2)}
\\[1ex]
\le
2^{\max (\frac 1p-1,0)} (\nm {T_1-T_{1,j}}{\maclB (\mascB _1,\mascB _2)}
+ \nm {T_2-T_{2,j}}{\maclB (\mascB _1,\mascB _2)}).
\end{multline*}
By taking the infimum on the right-hand side over all possible $T_{1,j}$ and
$T_{2,j}$ we obtain \eqref{Eq:SingQuBTriang}.

\par

By letting $j_1=j_2=j$ in \eqref{Eq:SingQuBTriang} and using the fact that
$\sigma _j(T)$ is non-increasing with respect to $j$ we get
\begin{multline*}
\nm {T_1+T_2}{\mascI _q}
\le
\left ( 2\sum _{j=0}^\infty \sigma _{2j+1}(T_1+T_2)^q \right )^{\frac 1q} 
\\[1ex]
\le
2^{\max (\frac 1p-1,0)+\frac 1q} \nm {\{ \sigma _j(T_1)+\sigma _j(T_2) \} _{j=1}^\infty }
{\ell ^q(\mathbf Z_+)}
\\[1ex]
\le
2^{\max (\frac 1p-1,0)+\max (\frac 1q-1,0)+\frac 1q}(
\nm {\{ \sigma _j(T_1)}{\ell ^q(\mathbf Z_+)}+\nm {\sigma _j(T_2) \} _{j=1}^\infty }
{\ell ^q(\mathbf Z_+)}
\\[1ex]
= 2^{\max (\frac 1p-1,0)+\max (\frac 1q-1,0)+\frac 1q}
(\nm {T_1}{\mascI _q}+\nm {T_2}{\mascI _q}).\qedhere
\end{multline*}
%%
%
%
%
%$$
%\sigma _j(T_1+T_2,\mascB _1,\mascB _2)^p
%\le
%\sigma _j(T_1,\mascB _1,\mascB _2)^p
%+
%\sigma _j(T_2,\mascB _1,\mascB _2)^p,
%$$
%and the result follows by taking the sum of all $j\ge 1$.
\end{proof}

\par

\begin{proof}[Proof of Proposition \ref{Prop:MatrixOpSchatten}]
We may assume that equality holds in \eqref{Eq:pqrCond}. Let 
$\mascB = \mascI _q(\ell ^\infty _{(\omega _1)}(J_1),\ell ^p _{(\omega _2)}(J_2))$
and $\mascB _2= \ell ^p_{(\omega _2)}(J_2)$. Then
$A_1+A_2\in \mascB$ when $A_1,A_2\in \mascB$, and \eqref{Eq:ShattQuBTriang}
shows that
$$
\nm {A_1+A_2}{\mascB}^r \le
\nm {A_1}{\mascB}^r+\nm {A_2}{\mascB}^r,\quad
A_1,A_2\in \mascB .
$$
Furthermore, if $A_{j_2,j_1}$ are the
same as in Proposition \ref{Prop:MatrixMinimal},
then $A_{j_2,j_1}\in \mascB$ since $A_{j_2,j_1}\in \mathbb U_0$. For
$f\in \ell ^\infty _{(\omega _1)}(J_1)$ with $\nm f{\ell ^\infty _{(\omega _1)}}\le 1$
we get
\begin{multline*}
\nm {A_{j_{0,2},j_{0,1}}f}{\ell ^p_{(\omega _2)}}
=
\left (
\sum _{j_2} \left |
\delta _{j_2,j_{0,2}}\sum _{j_1} \delta _{j_1,j_{0,1}}f(j_1)\omega _2(j_2)
\right |^p
\right )^{\frac 1p}
\\[1ex]
=
|f(j_{0,1})\omega _2(j_{0,2})| \le C\frac {\omega _2(j_{0,2})}{\omega _1(j_{0,1})}.
\end{multline*}
%%
%$$
%\nm {A_{j_0,k_0}f}{\ell ^p_{(\omega _2)}}^p
%=
%\sum _j \left |
%\delta _{j,j_0}\sum _k \delta _{k,k_0}f(k)\omega _2(j)
%\right |^p
%=
%|f(k)\omega _2(j)|^p \le \left (
%C\frac {\omega _2(j)}{\omega _1(k)}\right )^p.
%$$
From these estimates and Proposition \ref{Prop:MatrixMinimal} we get
$\mathbb U^r(\omega ,\boldsymbol J)\subseteq \mascB$, and the result follows.
\end{proof}

\par

\begin{proof}[Proof of Theorem \ref{Thm:QuasiBanachSchatten}]
By \eqref{calculitransform} and Proposition \ref{Prop:ExpOpSTFT}
we may assume that $A=0$.
Let
$$
\omega _0(x,\xi ,y,\eta )=\omega (x,\eta ,\xi -\eta,y-x).
$$
By Lemma 3.3
in \cite{Toft13}, there is a lattice $\Lambda \subseteq \rr d$,
$A\in \mathbb U^r(\omega _0,\Lambda ^2)$ and $\phi _1,\phi _2\in M^r_{(v)}(\rr d)$
such that $\op (a) = D_{\phi _1}\circ A\circ C_{\phi _2}$.
More refined,  by \cite{KSV} we may choose both $\phi _1$, $\phi _2$
and their dual windows to belong to $\Sigma _1(\rr d)$ (cf. Remark 
\ref{Rem:GaborAtomOrderLessThanOne}).

\par

We have
\begin{alignat}{2}
C_{\phi _2} &: & M^\infty _{(\omega _1)}(\rr d) &\to \ell ^\infty
_{(\omega _1)}(\Lambda ^2)\label{Eq:SamplCont}
\intertext{and}
D_{\phi _1} &: & \ell ^p _{(\omega _2)}(\Lambda ^2) &\to
M^p _{(\omega _2)}(\rr d)\label{Eq:SynthCont}
\end{alignat}
are continuous and
$$
A\in \mascI _q(\ell ^\infty _{(\omega _1)}(\Lambda ^2),
\ell ^p _{(\omega _2)}(\Lambda ^2))
$$
in view of Proposition \ref{Prop:MatrixOpSchatten}. Hence, by Proposition
\ref{Prop:NuclQBanach} (1) we get
$$
\op (a) = D_{\phi _1}\circ A\circ C_{\phi _2}
\in \mascI _q(M^\infty _{(\omega _1)}(\rr d),
M^p _{(\omega _2)}(\rr d)),
$$
which is the same as \eqref{Eq:SchattenPseudo}.
\end{proof}

\par

Theorem \ref{Thm:QuasiBanachSchatten} also leads to the following
Schatten-von Neumann result on operators with kernels in modulation
spaces, which in particular improve \cite[Corollary 3.1]{DeRuWa}. 

\par

\begin{thm}\label{Thm:KernelsSchatten}
Let $\omega _j\in \mascP _E(\rr {2d_j})$ for $j=1,2$ and $\omega \in
\mascP _E(\rr {2d_2+2d_1})$ be such that
$$
\frac {\omega _2(x,\xi )}{\omega _1(y,\eta )}\lesssim \omega (x,y,\xi ,-\eta ),
\qquad x,\xi \in \rr {d_2},\ y,\eta \in \rr {d_1},
$$
and let $p,q,r\in (0,\infty ]$ be such that \eqref{Eq:pqrCond} holds.
Also let $T$ be a linear and continuous operator from $\Sigma _1(\rr {d_1})$
to $\Sigma _1'(\rr {d_2})$ with distribution kernel
$K\in M^r_{(\omega )}(\rr {d_2+d_1})$.
Then
$$
T\in \mascI _q(M^\infty _{(\omega _1)}(\rr {d_1}),M^p_{(\omega _2)}(\rr {d_2}))
$$
and
$$
\nm T{\mascI _q(M^\infty _{(\omega _1)},M^p_{(\omega _2)})}\lesssim
\nm K{M^r_{(\omega )}}.
$$
\end{thm}

\par

\begin{proof}
If $d_1=d_2$, then the
result follows from Proposition 2.5 (2) in \cite{Toft16}
and Theorem \ref{Thm:QuasiBanachSchatten}. We need to consider the case
when $d_1\neq d_2$.

\par

If $d_2>d_1$, then let $d_0=d_2-d_1$, $\phi \in \Sigma _1(\rr {d_0})\setminus 0$
be fixed, and set
\begin{gather*}
K_0(x,y_1)= K(x,y)\phi (y_0),
\quad
\omega _0(x,y_1,\xi ,\eta _1)=\omega (x,y,\xi ,\eta)
\intertext{and}
\omega _{0,1}(y_1,\eta _1)=\omega _1(y,\eta )
\intertext{when}
y_1=(y,y_0)\in \rr {d_1}\times \rr {d_0} \simeq \rr {d_2}
\quad \text{and}\quad
\eta_1=(\eta ,\eta _0)\in \rr {d_1}\times \rr {d_0} \simeq \rr {d_2}.
\end{gather*}
Also let $T_0$ be the operator from $\maclS _{1/2}(\rr {d_2})$ to
$\maclS _{1/2}'(\rr {d_2})$ with kernel $K_0$. By straight-forward computations
it follows that $\nm K{M^r_{(\omega )}}\asymp \nm {K_0}{M^r_{(\omega _0)}}$
and
$$
\sigma _j(T; M^\infty _{(\omega _1)}(\rr {d_1}),M^p _{(\omega _2)}(\rr {d_2}))
\lesssim
\sigma _j(T_0; M^\infty _{(\omega _{0,1})}(\rr {d_2}),M^p _{(\omega _2)}(\rr {d_2})),
$$
giving that
$$
\nm {T}{\mascI _q(M^\infty _{(\omega _1)},M^p _{(\omega _2)})}
\lesssim
\nm {T_0}{\mascI _q(M^\infty _{(\omega _{0,1})},M^p _{(\omega _2)})}
\lesssim \nm {K_0}{M^r_{(\omega _0)}} \asymp
\nm K{M^r_{(\omega )}},
$$
and the result follows in this case.

\par

If instead $d_2<d_1$, then let $d_0=d_1-d_2$, $\phi \in \Sigma _1(\rr {d_0})\setminus 0$
be fixed, and set
\begin{gather*}
K_0(x_1,y)= \phi (x_0)K(x,y),
\quad
\omega _0(x_1,y,\xi _1,\eta )=\omega (x,y,\xi ,\eta)
\intertext{and}
\omega _{0,2}(x_1,\xi _1)=\omega _2(x,\xi )
\intertext{when}
x_1=(x,x_0)\in \rr {d_2}\times \rr {d_0} \simeq \rr {d_1}
\quad \text{and}\quad
\xi_1=(\xi ,\xi _0)\in \rr {d_2}\times \rr {d_0} \simeq \rr {d_1}.
\end{gather*}
By similar arguments as above we get
$$
\nm {T}{\mascI _q(M^\infty _{(\omega _1)},M^p _{(\omega _2)})}
\lesssim
\nm {T_0}{\mascI _q(M^\infty _{(\omega _1)},M^p _{(\omega _{0,2})})}
\lesssim \nm K{M^r_{(\omega )}},
$$
when $T_0$ is the operator with kernel $K_0$. This gives the result.
\end{proof}

\par

%%%%%%%%%%%%%%%%%%%%%%%
\section{Nuclearity properties for
operators with kernels in modulation spaces}
\label{sec4}
%%%%%%%%%%%%%%%%%%%%%%%

\par

In this section we perform analogous investigations as in the the previous section to
deduce $r$-nuclear properties of operators with kernels in $M^p_{(\omega )}$.
At the same time we deduce analogous properties for
pseudo-differential operators with symbols in $M^p_{(\omega )}$.

\par

First we have the following concerning $\mascN _p(\mascB _0,\mascB )$ in Section
\ref{sec1}.

\par

\begin{prop}\label{Prop:NuclQBanach}
Let $p\in (0,1]$, $\mascB _0$ be a Banach space and $\mascB$ be a
quasi-Banach space of order
$p$. Then $\mascN _p(\mascB _0,\mascB )$ is a quasi-Banach space of order $p$.
\end{prop}

\par

It is clear that $\mascN _p(\mascB _0,\mascB )$ in Proposition \ref{Prop:NuclQBanach}
is a quasi-normed space of order $p$. What remains to verify is that
$\mascN _p(\mascB _0,\mascB )$ is complete, and this follows by similar arguments as
in the proof of Proposition \ref{CalMpComplete}. The details are left for the reader.
(See also \cite{Gro}.)

\par

We have now the following results.

\par

\begin{thm}\label{Thm:QuasiBanachNuclear}
Let $A\in \GL (d,\mathbf R)$, $p\in (0,1]$
and let
$\omega _0\in \mascP (\rr {4d})$ and $\omega _1,\omega _2
\in \mascP _E(\rr {2d})$ be such that
\eqref{Eq:omega120Cond} holds true. Then
\begin{equation}\label{Eq:NuclearPseudo}
M^p_{(\omega _0)}(\rr {2d})\subseteq 
u_{A,p} (M^\infty _{(\omega _1)}(\rr d),M^p_{(\omega _2)}(\rr d) ).
\end{equation}
\end{thm}

\par

\begin{thm}\label{Thm:KernelsNuclear}
Let $\omega _j\in \mascP _E(\rr {2d_j})$ for $j=1,2$ and $\omega \in
\mascP _E(\rr {2d_2+2d_1})$ be such that
$$
\frac {\omega _2(x,\xi )}{\omega _1(y,\eta )}\lesssim \omega (x,y,\xi ,-\eta ),
\qquad x,\xi \in \rr {d_2},\ y,\eta \in \rr {d_1},
$$
and let $p\in (0,1]$. Also let $T$ be a linear and continuous operator
from $\maclS _{1/2}(\rr {d_1})$ to $\maclS _{1/2}'(\rr {d_2})$ with
distribution kernel $K\in M^p_{(\omega )}(\rr {d_2+d_1})$.
Then
$$
T\in \mascN _p(M^\infty _{(\omega _1)}(\rr {d_1}),M^p_{(\omega _2)}(\rr {d_2}))
$$
and
$$
\nm T{\mascN _p(M^\infty _{(\omega _1)},M^p_{(\omega _2)})}\lesssim
\nm K{M^p_{(\omega )}}.
$$
\end{thm}

\par

For the proofs of these results we again consider related questions for matrix operators.

\par

\begin{prop}\label{Prop:MatrixOpNuclear}
Let $p\in (0,1]$, $J_k$ be index set, $\omega _k$
be positive functions on $J_k$, $k=1,2$, let $\boldsymbol J = J_2\times J_1$ and suppose
$$
\frac {\omega _2(j_2)}{\omega _1(j_1)}\lesssim \omega (j_2,j_1),\quad (j_2,j_1)\in \boldsymbol J.
$$
Then $\mathbb U^p(\omega ,\boldsymbol J)\subseteq
\mascN _p (\ell ^\infty _{(\omega _1)}(J_1),\ell ^p_{(\omega _2)}(J_2))$.
\end{prop}

\par

\begin{proof}
The set $\mascN _p (\ell ^\infty _{(\omega _1)}(J_1),\ell ^p_{(\omega _2)}(J_2))$ is a
quasi-Banach space of order $p$. Furthermore, let $A_{j_2,j_1}$ be the same as in
Proposition \ref{Prop:MatrixMinimal} and let $\rho _k$ from $\mathbf Z_+$ to
$J_k$ be enumerations of $J_k$, $k=1,2$. Since
$$
\{ \delta _{j_{0,1},\rho _1(l)} \} _{l=1}^\infty \in \ell ^1_{(1/\omega _1\circ \rho _1)}(\mathbf Z_+)
\quad \text{and}\quad
\nm f{\ell ^1_{(1/\omega _1 \circ \rho _1)}} = \nm f{(\ell ^\infty _{(\omega _1 \circ \rho _1)})'}
$$
when $f\in \ell ^1_{(1/\omega _1 \circ \rho _1)}(\mathbf Z_+)$ we get
\begin{multline*}
\nm {A_{j_{0,2},j_{0,1}}}{\mascN _p (\ell ^\infty _{(\omega _1)}(J_1),
\ell ^p_{(\omega _2)}(J_2))}
\le
\nm {\{ \delta _{j_{0,1},\rho _1(l)} \} _{l=1}^\infty}{\ell ^1 _{(1/\omega _1\circ\rho _1)}}
\nm {\{ \delta _{j_{0,2},\rho _2(l)} \} _{l=1}^\infty}{\ell ^\infty _{(\omega _2\circ\rho _2)}}
\\[1ex]
\le
\frac {\omega _2(j_{0,2})}{\omega _1(j_{0,1})}
\lesssim
\omega (j_{0,2},j_{0,1})
\end{multline*}
The result now follows from Proposition \ref{Prop:MatrixMinimal}.
\end{proof}

\par

\begin{proof}[Proof of Theorem \ref{Thm:QuasiBanachNuclear}]
By \eqref{calculitransform} and Proposition \ref{Prop:ExpOpSTFT}
we may assume that $A=0$.
Let $\omega _0$, $A$, $\phi _1$, $\phi _2$ and $\Lambda$ be as in the proof
of Theorem \ref{Thm:QuasiBanachSchatten} with $p$ in place of $r$.

\par

Since $C_{\phi _2}$ and $D_{\phi _1}$ in \eqref{Eq:SamplCont}
and \eqref{Eq:SynthCont} are continuous and
$$
A\in \mascN _p(\ell ^\infty _{(\omega _1)}(\Lambda ^2),
\ell ^p _{(\omega _2)}(\Lambda ^2))
$$
by Proposition \ref{Prop:MatrixOpNuclear},
Proposition \ref{Prop:NuclQBanach} gives
$$
\op (a) = D_{\phi _1}\circ A\circ C_{\phi _2}
\in \mascN _p(M^\infty _{(\omega _1)}(\rr d),
M^p _{(\omega _2)}(\rr d)),
$$
which is the same as \eqref{Eq:NuclearPseudo}.
\end{proof}

\par

Finally, Theorem \ref{Thm:KernelsNuclear} now follows by similar arguments
as in the proof of Theorem \ref{Thm:KernelsSchatten}, where Theorem
\ref{Thm:QuasiBanachNuclear} is used instead of Theorem \ref{Thm:QuasiBanachSchatten}.
The details are left for the reader.

\par

\end{document}